\documentclass[11pt]{amsart}

\newtheorem{theorem}{Theorem}[section]

\newtheorem{lemma}[theorem]{Lemma}
\newtheorem{proposition}[theorem]{Proposition}

\theoremstyle{definition}

\theoremstyle{remark} \theoremstyle{remark}

\newtheorem{example}[theorem]{Example}

\numberwithin{equation}{section}

\begin{document}

\title[Legendrian submanifolds of $(\kappa,\mu)$-spaces]{A classification of totally geodesic and totally umbilical Legendrian submanifolds of $(\kappa,\mu)$-spaces\thanks{The first two authors are partially supported by the MINECO-FEDER grant MTM2014-52197-P. They are members of the IMUS (Instituto de Matem\'aticas de la Universidad de Sevilla), and of the PAIDI groups FQM-327 and FQM-226 (Junta de Andaluc\'ia, Spain), respectively.}
}
\author[A. Carriazo]{Alfonso Carriazo}
 \address{Departamento de Geometr\'{i}a y Topolog\'{i}a, c/ Tarfia s/n, Universidad de Sevilla, Sevilla 41012, Spain}
 \email{carriazo@us.es}

\author[V. Mart\'{i}n-Molina]{Ver\'onica  Mart\'{i}n-Molina}
 \address{Departamento de Did\'{a}ctica de las Matem\'{a}ticas, Facultad de Ciencias de la Educaci\'on, c/ Pirotecnia s/n, Universidad de Sevilla, Sevilla 41013, Spain}
 \email{veronicamartin@us.es}

\author[L. Vrancken]{Luc Vrancken}
 \address{LAMAV, Universit\'{e} de Valenciennes, 59313, Valenciennes Cedex 9, France \and Departement Wiskunde, KU Leuven, Celestijnenlaan 200 B, 3001, Leuven, Belgium}
 \email{luc.vrancken@univ-valenciennes.fr}

\begin{abstract}
We present classifications of totally geodesic and totally umbilical Legendrian submanifolds of $(\kappa,\mu)$-spaces with Boeckx invariant $I \leq -1$. In particular, we prove that such submanifolds must be, up to local isometries, among the examples that we explicitly construct.
\end{abstract}



\maketitle

\section{Introduction}

Although under a different name, $(\kappa,\mu)$-spaces were introduced by D. E. Blair, T. Koufogiorgos and B. J. Papantoniou in \cite{blair95} (for technical details, we refer to the Preliminaries section). Actually, these manifolds have proven to be really useful, because they provide non-trivial examples for some important classes of contact metric manifolds (for instance, the unit tangent sphere bundle of any Riemannian manifold of constant sectional curvature carries such a structure). The theory of $(\kappa,\mu)$-spaces was soon developed, with many interesting results. In particular, we can point out the outstanding paper \cite{boeckx2000}, where E. Boeckx classified non-Sasakian $(\kappa,\mu)$-spaces by using the invariant $I$ (depending only on the values of $\kappa$ and $\mu$) introduced by himself. He also provided examples for all possible $(\kappa,\mu)$.

Nevertheless, the theory of submanifolds of $(\kappa,\mu)$-spaces has not been developed in depth yet, even if we can find some very interesting papers about it. For example, in \cite{CTT}, B. Cappelletti Montano, L. Di Terlizzi and M. M. Tripathi proved that any invariant submanifold of a non-Sasakian contact $(\kappa,\mu)$-space is always totally geodesic and, conversely, that every totally geodesic submanifold of a non-Sasakian contact $(\kappa,\mu)$-space such that $\mu\neq 0$ and the characteristic vector field $\xi$ is tangent to the submanifold is invariant. Motivated by these results, we consider the case of submanifolds which are normal to $\xi$. Moreover, we restrict our study to the case of Legendrian submanifolds, i.e., those with dimension $n$ in a $(2n+1)$-dimensional ambient space.

From our point of view, a key step in continuing the analysis of submanifolds of $(\kappa,\mu)$-spaces should be to understand the behavior of the so-called $h$ operator of the ambient space with respect to the submanifold. Therefore, in this paper, we first establish in Section \ref{sec3} a decomposition of that operator in its tangent and normal parts, and find its main properties. In Section \ref{sec4} we present several examples of totally geodesic and totally umbilical Legendrian submanifolds of $(\kappa,\mu)$-spaces with $I\leq -1$. Actually, we prove in Section \ref{sec5} that these examples constitute the complete local classification of these kinds of submanifolds, given by our main results Theorems \ref{theo51} and \ref{theo52}.

\section{Preliminaries}

Let $M$ be a $(2n+1)$-dimensional smooth manifold $M$. Then an \textit{almost contact structure} is a triplet $(\varphi,\xi,\eta)$, where $\varphi$ is a $(1,1)$-tensor field, $\eta$ a $1$-form and $\xi$ a vector field on $M$ satisfying the following conditions
\begin{equation}\label{almostcontact}
\varphi^{2}=-I+\eta\otimes\xi, \; \eta(\xi)=1.
\end{equation}
It follows from \eqref{almostcontact} that $\varphi\xi=0$, $\eta\circ\varphi=0$ and that $\text{rank}(\varphi)=2n$ (\cite{blairbook}).

Any almost contact manifold $\left(M,\varphi,\xi,\eta\right)$ admits a \emph{compatible metric}, i.e. a Riemannian metric $g$ satisfying
\begin{equation*}
g\left(\varphi X,\varphi Y\right)=g\left(X,Y\right)-\eta\left(X\right)\eta\left(Y\right),
\end{equation*}
for all vector fields $X,Y$ on $M$. It follows that $\eta=g(\cdot,\xi)$ and $g(\cdot,\varphi\cdot)=-g(\varphi\cdot,\cdot)$. The manifold $M$ is said to be an \textit{almost contact metric manifold} with structure $\left(\varphi,\xi,\eta,g\right)$.

We can define the \textit{fundamental $2$-form} $\Phi$ of an almost contact metric manifold by $\Phi\left(X,Y\right)=g\left(X,\varphi Y\right)$. If $\Phi=d\eta$, then $\eta$ becomes a contact form, with $\xi$ its Reeb/characteristic vector field and ${\mathcal D}=\ker(\eta)$ its corresponding contact distribution, and $M(\varphi,\xi,\eta,g)$ is called a \textit{contact metric manifold}.

Every contact metric manifold satisfies
\begin{equation}\label{eq-h}
\nabla\xi=-\varphi -\varphi h,
\end{equation}
where $2h$ is the Lie derivative of $\varphi$ in the direction of $\xi$, i.e. $h=\frac12 L_\xi \varphi$. The tensor field $h$ is symmetric with respect to
$g$, satisfies $h\xi=0$, anticommutes with $\varphi$ and vanishes identically if and only if the Reeb vector field $\xi$ is Killing. In this last case the contact metric manifold is said to be
$K$-\textit{contact}.

An almost contact metric manifold is said to be \textit{normal} if $N_{\varphi}:=[\varphi,\varphi]+2d\eta\otimes\xi=0$.
A normal contact metric manifold is called a \textit{Sasakian manifold}. Any Sasakian manifold is K-contact and the converse holds  in dimension $3$ but not in general.


A special class of contact metric manifold is that of $(\kappa,\mu)$-spaces, first studied in \cite{blair95} under the name of \textit{contact metric manifolds with $\xi$ belonging to the $(\kappa,\mu)$-distribution}. A \textit{contact metric $(\kappa,\mu)$-space} is one satisfying the condition
\begin{equation}\label{kappamu}
R(X,Y)\xi=\kappa \, (\eta(Y)X-\eta(X)Y)+\mu \, (\eta(Y)hX-\eta(X)hY),
\end{equation}
for some constants $\kappa$ and $\mu$. In this paper, all manifolds will be contact metric, so we will shorten ``contact metric $(\kappa,\mu)$-space" to ``$(\kappa,\mu)$-space".

Every $(\kappa,\mu)$-space satisfies
\begin{align}
h^2&=(\kappa-1)\varphi^2, \label{eqh2} \\
(\nabla_X\varphi)Y&=g(X,Y+hY)\xi-\eta(Y)(X+hX), \label{nablaphi}\\
(\nabla_Xh)Y&=((1-\kappa)g(X,\varphi Y)-g(X,\varphi hY))\xi \nonumber\\
&\quad -\eta(Y)((1-\kappa)\varphi X+\varphi hX)-\mu\eta(X)\varphi hY. \label{nablah}
\end{align}
Moreover, we have the following result:
\begin{theorem}[\cite{blair95}] \label{theoblair}
Let $M^{2n+1}(\varphi,\xi,\eta,g)$ be a $(\kappa,\mu)$-space. Then $\kappa \leq 1$. If  $\kappa = 1$, then $h = 0$ and $M^{2n+1}$ is a Sasakian manifold. If $\kappa < 1$, $M^{2n+1}$ admits three mutually orthogonal and integrable distributions $E_M(0)=\text{span}(\xi)$, $E_M(\lambda)$ and $E_M(-\lambda)$ determined by
the eigenspaces of $h$, where $\lambda=\sqrt{1-\kappa}$.
\end{theorem}

As a consequence of this theorem, it was also proved in \cite{blair95} that the sectional curvature of a plane section $\{X,Y\}$ normal to $\xi$ is given by
\begin{equation}\label{sectional}
K(X,Y)=
\begin{cases}
2(1+\lambda)-\mu, \text{ for any } X,Y\in E_M(\lambda), \quad n>1,\\
2(1-\lambda)-\mu, \text{ for any } X,Y\in E_M(-\lambda), \quad n>1,\\
-(\kappa+\mu)(g(X,\varphi Y))^2, \text{ for any unit vectors } X\in E_M(\lambda), Y\in E_M(-\lambda).
\end{cases}
\end{equation}

Given a contact metric manifold $M^{2n+1}(\varphi,\xi,\eta,g)$, a \textit{$D_a$-homothetic deformation} is a change of structure tensors of the form
\begin{equation}\label{deformation}
\tilde\varphi= \frac1a \varphi, \; \tilde\xi=\xi, \; \tilde\eta=a \eta, \; \tilde g = a g +a(a-1) \eta \otimes \eta,
\end{equation}
where $a$ is a positive constant. It is well known that $M^{2n+1}(\tilde\varphi,\tilde\xi,\tilde\eta,\tilde{g})$ is also a
contact metric manifold.

It was also proved in \cite{blair95} that the class of $(\kappa,\mu)$-spaces remains invariant under $D_a$-homothetic deformations. Indeed, applying one of these deformations to a $(\kappa,\mu)$-space yields a new $(\tilde\kappa,\tilde\mu)$-space, where
\[
\tilde\kappa =\frac{\kappa+a^2-1}{a^2} , \, \tilde\mu = \frac{\mu+2a-2}{a}.
\]

Many authors studied $(\kappa,\mu)$-spaces later, as can be seen in \cite{blairbook}. We highlight here the work of Boeckx, who gave in  \cite{boeckx2000} an explicit writing of the curvature tensor of these spaces:
\begin{equation}\label{Rkappamu}
\begin{aligned}
R(X,Y)Z &=\left(1-\frac\mu2 \right)(g(Y,Z)X-g(X,Z)Y)\\
        &+g(Y,Z)hX-g(X,Z)hY-g(hX,Z)Y+g(hY,Z)X\\
        &+\frac{1-\frac\mu2}{1-\kappa}(g(hY,Z)hX-g(hX,Z)hY)\\
        &-\frac\mu2 (g(\varphi Y,Z)\varphi X-g(\varphi X,Z)\varphi Y)+\mu g(\varphi X,Y)\varphi Z\\
        &+\frac{\kappa-\frac\mu2}{1-\kappa} (g(\varphi hY,Z)\varphi h X-g(\varphi hX,Z)\varphi h Y)\\
        &-\eta(X)\eta(Z)\left(\left(\kappa-1+\frac\mu2 \right)Y+(\mu-1)hY\right)\\
        &+\eta(Y)\eta(Z)\left(\left(\kappa-1+\frac\mu2 \right)X+(\mu-1)hX \right)\\
        &+\eta(X)\left(\left(\kappa-1+\frac\mu2 \right)g(Y,Z)+(\mu-1)g(hY,Z)\right)\xi\\
        &-\eta(Y)\left(\left(\kappa-1+\frac\mu2 \right)g(X,Z)+(\mu-1)g(hX,Z)\right)\xi.
\end{aligned}
\end{equation}
Boeckx \cite{boeckx2000} also classified the $(\kappa,\mu)$-spaces in terms of an invariant that he introduced: $I_M=\frac{1-\frac{\mu}2}{\sqrt{1-\kappa}}$. Indeed, he proved that if $M_1$ and $M_2$ are two non-Sasakian $(\kappa_i,\mu_i)$-spaces of the same dimension, then $I_{M_1}=I_{M_2}$ if and only if, up to a $D_a$-homothetic deformation, the two spaces are locally isometric as contact metric spaces. In particular, if both spaces are simply connected and complete, they are globally isometric up to a $D_a$-homothetic deformation.

It was also stated in paper \cite{boeckx2000} that ``it follows that we know all non-Sasakian $(\kappa,\mu)$-spaces locally as soon as we have, for every odd dimension $2n + 1$ and for every possible value for the invariant $I$, one $(\kappa,\mu)$-space $M$ with $I_M=I$.''
For $I > -1$, we have the unit tangent sphere bundle $T_1 M^n(c)$ of a space of constant curvature $c$ ($c\neq1$) for the appropriate $c$ (see \cite{blair95}). For $I \leq -1$, Boeckx presented in \cite{boeckx2000} the following examples for any possible odd dimension $2n + 1$ and value of $I$.

\begin{example}[\cite{boeckx2000}] \label{books}
Let $\mathfrak{g}$ be a $(2n+1)$-dimensional Lie algebra with basis $\{\xi,X_1,\dots,X_n,$ $Y_1,\dots,Y_n\}$ and the Lie brackets given by
\begin{equation} \label{brackets-books}
\begin{array}{rclrcl}
\left[\xi,X_1\right]  & = &  -\displaystyle{\frac{\alpha \beta}{2}X_2-\frac{\alpha^2}{2}Y_1}, & \left[Y_i,Y_j\right]  & = & 0, \quad i,j \neq 2, \\
\\
\left[\xi,X_2\right]  & = &  \displaystyle{\frac{\alpha \beta}{2}X_1-\frac{\alpha^2}{2}Y_2}, & \left[X_1,Y_1\right]  & = & -\beta X_2+2\xi, \\
\\
\left[\xi,X_i\right]  & = & -\displaystyle{\frac{\alpha^2}{2}Y_i}, \quad i\geq 3, & \left[X_1,Y_i\right]  & = & 0, \quad i\geq 2, \\
\\
\left[\xi,Y_1\right]  & = & \displaystyle{\frac{\beta^2}{2}X_1-\frac{\alpha\beta}{2}Y_2}, & \left[X_2,Y_1\right]  & = & \beta X_1-\alpha Y_2, \\
\\
\left[\xi,Y_2\right]  & = & \displaystyle{\frac{\beta^2}{2}X_2+\frac{\alpha\beta}{2}Y_1}, & \left[X_2,Y_2\right]  & = & \alpha Y_1+2\xi, \\
\\
\left[\xi,Y_i\right]  & = & \displaystyle{\frac{\beta^2}{2}X_i}, \quad i\geq 3, & \left[X_2,Y_i\right]  & = & \beta X_i, \quad i\geq 3,\\
\\
\left[X_1,X_i\right]  & = & \alpha X_i, \quad i\neq 1, & \left[X_i,Y_1\right]  & = & -\alpha Y_i, \quad i\geq 3,\\
\\
\left[X_i,X_j\right]  & = & 0, \quad i,j\neq 1,& \left[X_i,Y_2\right]  & = & 0, \quad i\geq 3, \\
\\
\left[Y_2,Y_i\right]  & = & \beta Y_i, \quad i\neq 2, & \left[X_i,Y_j\right]  & = & \delta_{ij}(-\beta X_2+\alpha Y_1+2\xi), \quad i,j\geq 3,
\end{array}
\end{equation}
for real numbers $\alpha$ and $\beta$.
Next we define a left-invariant contact metric structure $(\varphi,\xi,\eta,g)$ on the associated Lie group $G$ as follows:
\begin{itemize}
\item the basis $\{\xi, X_1,\ldots, X_n, Y_1,\ldots, Y_n \}$ is orthonormal,

\item the characteristic vector field is given by $\xi$,

\item the one-form $\eta$ is the metric dual of $\xi$,

\item the $(1, 1)$-tensor field $\varphi$ is determined by
\(
\varphi \xi=0, \; \varphi X_i = Y_i, \; \varphi Y_i =-X_i.
\)
\end{itemize}
It can also be proved that $G$ is a $(\kappa,\mu)$-space with
\[
\kappa= 1-\frac{(\beta^2-\alpha^2)^2}{16}, \; \mu=2+\frac{\alpha^2+\beta^2}{2}.
\]
Moreover, supposing $\beta^2 >\alpha^2$ gives us that $\lambda=\frac{\beta^2-\alpha^2}4 \neq0$ and thus the $(\kappa,\mu)$-space is not Sasakian. The orthonormal basis also satisfies that $hX_i=\lambda X_i$ and $hY_i=-\lambda Y_i$.

Finally, $I_G= -\frac{\beta^2+\alpha^2}{\beta^2-\alpha^2} \leq -1$, so for the appropriate choice of $\beta > \alpha  \geq 0$, $I_G$ attains any real value smaller than or equal to $-1$.
\end{example}

\medskip

Lastly, we will recall some formulas from submanifolds theory in order to fix our notation. Let $N$ be an $n$-dimensional submanifold isometrically immersed in an $m$-dimensional Riemannian manifold $(M,g)$. Then, the
Gauss and Weingarten formulas hold:
\begin{align}
\nabla_XY&=\overline\nabla_XY+\sigma(X,Y), \label{eqgauss} \\
\nabla_XV&=-A_VX+\nabla^{\perp}_XV, \label{eqwein}
\end{align}
for any tangent vector fields $X,Y$ and any normal vector field $V$. Here $\sigma$ denotes the \textit{second fundamental form}, $A$ the \textit{shape operator} and $\nabla^{\perp}$ the \textit{normal connection}. It is well known that the second fundamental form and the shape operator are related the following way:
\begin{equation} \label{eqsigma}
g(\sigma(X,Y),V)=g(A_VX,Y).
\end{equation}
We denote by $R$ and $\overline R$ the curvature tensors of $M$ and
$N$, respectively. They are related by Gauss and Codazzi's
equations
\begin{equation}\label{eqgauss2}
R(X,Y,Z,W)=\overline R(X,Y,Z,W)-g(\sigma(X,W),\sigma(Y,Z))+g(\sigma(X,Z),\sigma(Y,W)),
\end{equation}
\begin{equation}\label{eqcodazzi}
(R(X,Y)Z)^\perp =(\nabla _{X}\sigma)(Y,Z)-(\nabla_{Y}\sigma)(X,Z),
\end{equation}
respectively, where $R(X,Y)Z^{\perp}$ denotes the normal
component of $R(X,Y)Z$ and
\begin{equation}
(\nabla _{X}\sigma)(Y,Z)=\nabla^\perp _{X}(\sigma(Y,Z))
-\sigma(\overline\nabla _{X}Y,Z)-\sigma(Y,\overline\nabla _{X}Z).
\end{equation}

The submanifold $N$ is said to be \textit{totally geodesic} if the second
fundamental form $\sigma$ vanishes identically. It is said that it is \textit{totally umbilical} if there exists a normal vector field $V$ such that $\sigma(X,Y)=g(X,Y)V$, for any tangent vector fields $X,Y$. In fact, it can be proved that, in such a case, $V$ has to be the \textit{mean curvature} $\widetilde H=\frac1n\sum_{i=1}^n\sigma(e_i,e_i)$, where $\{e_1,\dots,e_n\}$ is a local orthonormal frame. It is clear that every totally geodesic submanifold is also totally umbilical but the converse is not true in general.

\section{Decomposition of the $h$ operator} \label{sec3}

Let $N$ be a Legendrian submanifold of a $(2n+1)$-dimensional $(\kappa,\mu)$-space $M$, that is, an $n$-dimensional submanifold such that $\xi$ is normal to $N$. Therefore, $\eta(X)=0$ for any tangent vector field $X$ and so it follows from \eqref{almostcontact} that $\varphi^2X=-X$. Moreover, it was proved in \cite{lotta} that $N$ is an anti-invariant submanifold, i.e., $\varphi X$ is normal for any tangent vector field $X$. Moreover, under our assumptions about the dimensions of $M$ and $N$, it holds that every normal vector field $V$ can be written as $\varphi X$, for a certain tangent vector field $X$.

Therefore, we can decompose the $h$ operator in the following way:
\begin{equation} \label{h1h2}
hX=h_1X+\varphi h_2X,
\end{equation}
for any tangent vector field $X$, where $h_1X$ (respectively $\varphi h_2X$) denotes the tangent (resp. normal) component of $hX$.

We can prove the following properties:
\begin{proposition}
Let $N$ be a Legendrian submanifold of a $(\kappa,\mu)$-space $M$. Then, $h_1$ and $h_2$ are symmetric operators that satisfy $h_1\xi=h_2\xi=0$ and equations
\begin{align}
h_1^2+h_2^2&=(1-\kappa)I,\label{propo11}\\
h_1h_2&=h_2h_1.\label{propo12}
\end{align}
\end{proposition}
\begin{proof}
The symmetry of $h_1$ and $h_2$ can be directly obtained from that of $h$ and the compatibility of the metric $g$. Similarly, $h\xi=0$ implies $h_1\xi=h_2\xi=0$.

Furthermore, given a tangent vector field $X$, it follows from \eqref{almostcontact}, \eqref{h1h2} and the anticommutativity of $h$ and $\varphi$ that \begin{equation} \label{anticom}
h\varphi X=-\varphi hX=-\varphi h_1X+h_2X.
\end{equation}
Using \eqref{eqh2}, we have that $h^2X=(1-\kappa)X$. On the other hand, by virtue of \eqref{h1h2} and \eqref{anticom}, we obtain
\[
h^2X=h(h_1X+\varphi h_2X)=h_1^2X+\varphi h_2h_1X-\varphi h_1h_2X+h_2^2X.
\]
Joining both expressions for $h^2$ and identifying the tangent and normal parts give us equations \eqref{propo11} and \eqref{propo12}.
\end{proof}

\begin{proposition}
Let $N$ be a Legendrian submanifold of a $(\kappa,\mu)$-space $M$. Then, $h_1$ and $h_2$ satisfy
\begin{align}
(\overline\nabla_X h_1)Y&=-\varphi \sigma(X,h_2Y)-h_2\varphi \sigma(X,Y), \label{nablah1}\\
(\overline\nabla_X h_2)Y&=\; \; \, \varphi \sigma(X,h_1Y)+h_1\varphi \sigma(X,Y), \label{nablah2}
\end{align}
for any tangent vector fields $X,Y$.
\end{proposition}

\begin{proof}
It follows from Gauss and Weingarten formulas \eqref{eqgauss} and \eqref{eqwein} that
\[
(\nabla_X\varphi)Y=\nabla_X\varphi Y-\varphi \nabla_XY=-A_{\varphi Y}X+\nabla^{\perp}_X\varphi Y-\varphi\overline\nabla_XY-\varphi \sigma(X,Y),
\]
for any tangent vector fields $X,Y$. Therefore, by using \eqref{nablaphi} and identifying the tangent and normal components, we obtain:
\begin{align}
A_{\varphi Y}X &=-\varphi\sigma(X,Y), \label{wein1} \\
\nabla^{\perp}_X\varphi Y &=\varphi\overline\nabla_XY+g(X,Y+h_1Y)\xi. \label{wein2}
\end{align}

On the other hand, using \eqref{nablah} and \eqref{h1h2}, we have
\[
\nabla_X(h_1Y+\varphi h_2Y)-h(\nabla_X Y)=g(X,h_2Y)\xi,
\]
from where, by virtue of Gauss and Weingarten formulas \eqref{eqgauss} and \eqref{eqwein}, we deduce
\begin{equation} \label{wein3}
\overline\nabla_Xh_1Y+\sigma(X,h_1Y)-A_{\varphi h_2Y}X+\nabla^{\perp}_X\varphi h_2Y-h\overline\nabla_XY-h\sigma(X,Y)=g(X,h_2Y)\xi.
\end{equation}
We can put $h\overline\nabla_XY=h_1\overline\nabla_XY+\varphi h_2\overline\nabla_XY$ by \eqref{h1h2}. Now, by using \eqref{almostcontact}, we can write $\sigma(X,Y)=-\varphi^2\sigma(X,Y)+\eta(\sigma(X,Y))\xi$, and hence
$h\sigma(X,Y)=-h\varphi^2\sigma(X,Y)=\varphi h \varphi\sigma(X,Y)$. Again, equation \eqref{h1h2} gives us $h\sigma(X,Y)=\varphi h_1\varphi\sigma(X,Y)-h_2\varphi\sigma(X,Y)$. Therefore, if we substitute these two expressions, together with \eqref{wein1} and \eqref{wein2}, in \eqref{wein3}, we obtain:
\begin{equation} \label{wein4}
\begin{aligned}
\overline\nabla_Xh_1Y &+\sigma(X,h_1Y)+\varphi\sigma(X,h_2Y)+\varphi\overline\nabla_Xh_2Y+g(X,h_2Y+h_1h_2Y)\xi \\
&-h_1\overline\nabla_XY-\varphi h_2\overline\nabla_XY
-\varphi h_1\varphi\sigma(X,Y)+h_2\varphi\sigma(X,Y)=g(X,h_2Y)\xi.
\end{aligned}
\end{equation}
By identifying the tangent and normal parts of \eqref{wein4}, equations \eqref{nablah1} and \eqref{nablah2} hold.
\end{proof}

It is clear that, if we multiply \eqref{wein4} by $\xi$, then we obtain
\[
g(\sigma(X,h_1Y),\xi)+g(X,h_1h_2Y)=0,
\]
for any tangent vector fields $X,Y$. In fact, we can prove a more general result, which will be very useful in the proof of our main theorems:

\begin{lemma}
Let $N$ be a Legendrian submanifold of a $(\kappa,\mu)$-space $M$. Then,
\begin{equation} \label{eqsigma2}
g(\sigma(X,Y),\xi)+g(X,h_2Y)=0,
\end{equation}
for any tangent vector fields $X,Y$.
\end{lemma}

\begin{proof}
It follows from Weingarten equation \eqref{eqwein} and from \eqref{eqsigma} that
\[
g(X,\nabla_X\xi)+g(\sigma(X,Y),\xi)=0,
\]
for any tangent vector fields $X,Y$. Then, it is enough to use \eqref{almostcontact}, \eqref{eq-h} and \eqref{h1h2} to obtain \eqref{eqsigma2}.
\end{proof}

\section{Examples} \label{sec4}

We will present in this section some examples of totally geodesic and totally umbilical Legendrian submanifolds of the $(\kappa,\mu)$-spaces of Example \ref{books}. Let us begin with the totally geodesic ones.

\begin{example} \label{exx}
Let $M$ be a $(\kappa,\mu)$-space from Example \ref{books} with invariant $I_M\leq -1$. Then, the distribution $\mathcal{D}$ spanned by $\{X_1,\dots,X_n\}$ is involutive and any integral submanifold $N$ of it is a totally geodesic submanifold of $M$. Indeed, the involutive condition can be easily checked from \eqref{brackets-books}. In order to prove the totally geodesic one, it is enough to show that $\nabla_{X_i}X_j\in \mathcal{D}$, for any $i,j=1,\dots,n$, where $\nabla$ denotes the Levi-Civita connection on $M$. In fact, in can be directly computed that:
\begin{equation} \label{nablax}
\begin{aligned}
& \nabla_{X_1}X_1=\nabla_{X_1}X_2=0, \quad \nabla_{X_2}X_1=-\alpha X_2, \quad \nabla_{X_2}X_2=\alpha X_1, \\
& \nabla_{X_1}X_i=\nabla_{X_2}X_i=0, \mbox{ for any } i=3,\dots,n, \\
& \nabla_{X_i}X_1=-\alpha X_i, \quad \nabla_{X_i}X_2=0, \quad \nabla_{X_i}X_j=\delta_{ij}\alpha X_1, \mbox{ for any }i,j=3,\dots,n.
\end{aligned}
\end{equation}
Moreover, since $hX_i=\lambda X_i$ for any $i=1,\dots,n$, then $TN=E_M(\lambda)$.
\end{example}

\

\begin{example} \label{exy}
Let $M$ be a $(\kappa,\mu)$-space from Example \ref{books} with invariant $I_M\leq -1$. Then, the distribution $\mathcal{D}$ spanned by $\{Y_1,\dots,Y_n\}$ is also involutive and any integral submanifold $N$ of it is a totally geodesic submanifold of $M$. Indeed, both conditions can be checked the same way as in Example \ref{exx}, by taking now into account that:
\begin{equation} \label{nablay}
\begin{aligned}
& \nabla_{Y_1}Y_1=\beta Y_2, \quad \nabla_{Y_1}Y_2=-\beta Y_1, \quad \nabla_{Y_2}Y_1=\nabla_{Y_2}Y_2=0, \\
& \nabla_{Y_1}Y_i=\nabla_{Y_2}Y_i=0, \mbox{ for any } i=3,\dots,n, \\
& \nabla_{Y_i}Y_1=0, \quad \nabla_{Y_i}Y_2=-\beta Y_i, \quad \nabla_{Y_i}Y_j=\delta_{ij}\beta Y_2, \mbox{ for any }i,j=3,\dots,n.
\end{aligned}
\end{equation}
In this case, since $hY_i=-\lambda Y_i$ for any $i=1,\dots,n$, then $TN=E_M(-\lambda)$.
\end{example}

\

\begin{example} \label{exxy}
Let $M$ be a $(\kappa,\mu)$-space from Example \ref{books} with invariant $I_M\leq -1$. Then, the distribution $\mathcal{D}$ spanned by $\{X_1,Y_2,Z_3,\dots,Z_n\}$, where $Z_i$ is either $X_i$ or $Y_i$, for any $i=3,\dots,n$, is also involutive and any integral submanifold $N$ of it is a totally geodesic submanifold of $M$. Indeed, both conditions can be checked the same way as in Examples \ref{exx} and \ref{exy}, by using now \eqref{nablax}, \eqref{nablay} and the following formulas:
\begin{equation} \label{nablaz}
\begin{aligned}
& \nabla_{X_1}Y_i=0 \mbox{ for any }i=2,\dots,n, \\
& \nabla_{Y_2}X_i=0 \mbox{ for any }i=1,3,\dots,n, \\
& \nabla_{X_i}Y_2=\nabla_{Y_i}X_1=0 \mbox{ for any } i=3,\dots,n, \\
& \nabla_{X_i}Y_j=\nabla_{Y_i}X_j=0 \mbox{ for any }i,j=3,\dots,n, \mbox{ such that } i\neq j.
\end{aligned}
\end{equation}
Finally, if we define $E(\pm\lambda):=E_M(\pm \lambda) \cap N$, we can write
$TN=E(\lambda)\oplus E(-\lambda)$, with $\dim{E(\lambda)}=k$ (respectively $\dim{E(-\lambda)}=n-k$), where $k-1$ (resp. $n-k-1$) is the number of $Z_i$ such that $Z_i=X_i$ (resp. $Z_i=Y_i$). Therefore, we can obtain an example for any value of $k$ from $1$ to $n-1$.
\end{example}

We now present the family of totally umbilical examples:

\begin{example} \label{excd}
Let $M$ be a $(\kappa,\mu)$-space from Example \ref{books} with invariant $I_M\leq -1$. Then, the distribution $\mathcal{D}$ spanned by $\{cX_1+dY_1,\dots,cX_n+dY_n\}$, with $c,d$ non-zero constants, is involutive and any integral submanifold $N$ of it is a totally umbilical submanifold of $M$. Indeed, the involutive condition can be easily checked from \eqref{brackets-books}. In order to prove the totally umbilical one, we will first show that
$\sigma(cX_i+dY_i,cX_j+dY_j)=2\delta_{ij}cd\lambda\xi$
 by checking that the Levi-Civita connection on $M$ satisfies $\nabla_{cX_i+dY_i}(cX_j+dY_j)=Z+ 2\delta_{ij} cd\lambda  \xi $, with $Z \in \mathcal{D}$, for any $i,j=1,\dots,n$. In fact, it can be directly computed that:
\begin{equation*}
\begin{aligned}
\nabla_{cX_1+dY_1}(cX_1+dY_1)   &=\beta d(cX_2+dY_2)+2cd\lambda \xi , \\ \nabla_{cX_1+dY_1}(cX_2+dY_2)   &=-\beta d(cX_1+dY_1),\\
\nabla_{cX_2+dY_2}(cX_1+dY_1)  &=-\alpha c(cX_2+dY_2) ,             \\
\nabla_{cX_2+dY_2}(cX_2+dY_2)   &=\alpha c (cX_1+dY_1)+2cd\lambda \xi,\\
\nabla_{cX_1+dY_1}(cX_j+dY_j)   &=\nabla_{cX_2+dY_2}(cX_j+dY_j)=0, \text{ for any }j=3,\dots,n, \\
 \nabla_{cX_i+dY_i}(cX_1+dY_1) &=-\alpha c (cX_i+dY_i) ,  \\
 \nabla_{cX_i+dY_i}(cX_2+dY_2) & =-\beta d (cX_i+dY_i),  \text{ for any }i=3,\dots,n,\\
\nabla_{cX_i+dY_i}(cX_j+dY_j)&=\delta_{ij}
(\alpha c (cX_1+dY_1)+\beta d (cX_2+dY_2)+2cd\lambda \xi),\\
& \quad \text{ for any }i,j=3,\dots,n.
\end{aligned}
\end{equation*}
Therefore, we can write $\sigma(cX_i+dY_i,cX_j+dY_j)=g(cX_i+dY_i, cX_j+dY_j)\frac{2 cd\lambda}{c^2+d^2} \xi$ and, since $\frac{2 cd\lambda}{c^2+d^2} \xi\neq0$, the submanifold is totally umbilical but not totally geodesic.

Finally, we observe that $cX_i+dY_i$, $i=1,\ldots,n$, is not an eigenvector of $h$.
\end{example}

\section{Main results} \label{sec5}

\begin{theorem} \label{theo51}
Let $N$ be a Legendrian submanifold of a $(2n+1)$-dimensional $(\kappa,\mu)$-space $M$, with $\kappa<1$ and $I_M\leq -1$. If $N$ is totally geodesic, then, up to local isometries, it must be one of the submanifolds given in Examples \ref{exx}, \ref{exy} or \ref{exxy}.
\end{theorem}

\begin{proof}
Since the submanifold $N$ is totally geodesic, if follows directly from \eqref{eqsigma2} that $h_2=0$ and so $h|_N=h_1$ and $h_1^2=(1-\kappa)I$ (see \eqref{h1h2} and \eqref{propo11}). The operator $h_1$ is differentiable and symmetric, so it is diagonalisable and it has two eigenvalues $\pm\lambda=\pm \sqrt{1-\kappa}$, which are distinct and constant everywhere.

Let us denote by $E(\lambda)$ and $E(-\lambda)$ the eigenspaces of $h_1$ in $TN$ and by $k$ the dimension of $E(\lambda)$. This means that $\dim(E(-\lambda))=n-k$ (because $\dim N=n$) and that  $k\in\{0,\dots,n\}$. The multiplicities of both eigenspaces must be the same at every point because the coefficients of the characteristic polynomial are differentiable. Indeed, the characteristic polynomial of $h_1$ is completely determined by $k$ (thus, for different indices $k$, we get a different characteristic polynomial). Since $k$ is an integer, it is impossible by continuity to go from one to the other one, thus the eigendistributions are differentiable. We can then write
\begin{equation} \label{decomp}
TN=E(\lambda)\oplus E(-\lambda),
\end{equation}
where $\dim(E(\lambda))=k$ and $\dim(E(-\lambda))=n-k$, for a certain $k\in\{0,\dots,n\}$.

Moreover, we deduce from \eqref{nablah1} that $\overline\nabla h_1=0$. Therefore, it is straightforward to check that, if $Y_{\lambda} \in E(\lambda)$, then $\overline \nabla_X Y_{\lambda} \in E(\lambda)$, for every tangent vector field $X$. Similarly, if $Y_{-\lambda} \in E(-\lambda)$, then $\overline \nabla_X Y_{-\lambda} \in E(-\lambda)$. Thus, $E(\lambda)$ and $E(-\lambda)$ are parallel and hence involutive. By virtue of
Theorem 5.4 of \cite{KN}, $N$ can be locally decomposed as $M_1\times M_2$, where $M_1$ and $M_2$ are leaves of the distributions $E(\lambda)$ and $E(-\lambda)$, respectively. Furthermore, it follows from \eqref{sectional} that, if $\dim{M_1}\geq 2$ (resp. $\dim{M_2}\geq 2$), then $M_1$ (resp. $M_2$) has constant curvature $2(1+\lambda)-\mu=2\lambda(I_M+1)\leq 0$ (resp. $2(1-\lambda)-\mu=2\lambda(I_M-1)< 0$).

Recall that we have examples of submanifolds with decomposition \eqref{decomp} for every value of $k$. Indeed, see Example \ref{exx} for $k=n$, Example \ref{exy} for $k=0$ and Example \ref{exxy} for any value of $k$ from $1$ to $n-1$. Now, we will prove that any example must be one of these, up to local isometries.

Let us denote by $F:N^n\to M^{2n+1}(\kappa,\mu)$ the immersion of $N$ into $M$. Since $\kappa<1$ and $I_M\leq -1$, we can suppose that, locally, $M^{2n+1}(\kappa,\mu)$ is one of the Lie groups from Example \ref{books}. Thus, it is homogeneous and we can fix a point $p_0 \in N$ such that $F(p_0)=e$, where $e$ is the neutral element of the group.

We will give the explicit details when $2 \le k \le n-2$. The other cases can be done in a similar way. We have that $N = M_1( 2 \lambda (I_M+1)) \times M_2 (2 \lambda (I_M-1))$ and we also identify $N$ with its image as the (totally geodesic) integral submanifold through $e$ of the distribution spanned by $X_1,X_3,\dots,X_{k+1}, Y_2, Y_{k+2},\dots Y_n$. We denote by $G$ the latter immersion of $N$ and we pick an orthonormal basis $\{e_1,\ldots,e_n\}$ at the point $p_0$ of $N$, with $G(p_0)=e$, such that $E_{p_0}(\lambda)=\langle e_1(p_0),\ldots,e_k(p_0) \rangle$, $E_{p_0}(-\lambda)=\langle e_{k+1}(p_0),\ldots,e_n(p_0) \rangle$ and
\begin{align*}
&dG(e_1(p_0))= X_1(e),\\
&dG(e_j(p_0))=X_{j+1}(e), \quad j=2,\dots,k,\\
&dG(e_{k+1}(p_0)) =Y_2(e),\\
&dG(e_j(p_0)) = Y_j(e),\quad j=k+2,\dots,n,
\end{align*}

Note that by construction both
\[
X_1(e),X_3(e),\dots,X_{k+1}(e), \varphi Y_2(e),\varphi Y_{k+2}(e),\dots, \varphi Y_n(e)
\]
and
\[
dF(e_1(p_0)),\dots,dF(e_k(p_0)),\varphi dF(e_{k+1}(p_0)),\dots,\varphi dF(e_{n}(p_0))
\]
are basis of $E_e(\lambda)$. So, in view of Theorem 3 of \cite{boeckx2000}, there exists an isometry $H$ of $M^{2n+1}(\kappa,\mu)$ preserving the structure such that $H(e)=e$ and $H$ maps one basis of $E_e(\lambda)$ into the other one.
As a consequence, we have that $H \circ F(e) = G(e)$ and $d(H \circ F)(e_i) = dG(e_i)$.

We now take a geodesic $\gamma$ in $N$ through the point $p_0$. Since $N$ is totally geodesic, both with respect to the immersions $H \circ F$ and $G$, the curves $H \circ F(\gamma)$ and $G(\gamma)$ are both geodesics in $M^{2n+1}(\kappa,\mu)$ through $e$. Since $d(H \circ F)(e_i) = dG(e_i)$, they are also determined by the same initial conditions. Therefore, both curves need to coincide, so  $H \circ F(\gamma(s))= G(\gamma(s))$ for all $s$ and thus $F$ and $G$ are congruent.
\end{proof}

\begin{theorem} \label{theo52}
Let $N$ be a Legendrian submanifold of a $(2n+1)$-dimensional $(\kappa,\mu)$-space $M$, with $n\geq 3$, $\kappa<1$ and $I_M\leq -1$. If $N$ is totally umbilical (but not totally geodesic), then, up to local isometries, it must be one of the submanifolds given in Example \ref{excd}.
\end{theorem}

\begin{proof}
Since  $N$ is totally umbilical (but not totally geodesic), then there exists a normal vector field $V \neq0$ such that $\sigma(X,Y)=g(X,Y) V$. It follows from \eqref{eqsigma2} that $g(X,Y)\eta(V)+g(X,h_2Y)=0$, for any tangent vector fields $X,Y$, and thus
\begin{equation} \label{totumb}
h_2Y=aY,
\end{equation}
with $a=-\eta(V)$.

We will now prove that $a\neq0$. Indeed, if we suppose that $a=0$, then $h_2=0$ and, as in the proof of Theorem \ref{theo51}, we have that $h=h_1$, $h_1^2=(1-\kappa)I$ and $\overline{\nabla} h_1=0$. Moreover, since $h_2=0$, it is clear that $\overline{\nabla} h_2=0$ and we obtain from \eqref{nablah2} that $\varphi\sigma(X,h_1Y)+h_1\varphi\sigma(X,Y)=0$, which, by using that $N$ is totally umbilical, becomes
\begin{equation} \label{proof52}
g(X,h_1Y)\varphi V+g(X,Y)h_1\varphi V=0,
\end{equation}
for any tangent vector fields $X,Y$. Let us now choose unit vector fields $X_{\lambda} \in E(\lambda)$ and $X_{-\lambda} \in E(-\lambda)$. Then, taking $X=Y=X_{\lambda}$ in \eqref{proof52} implies
$h_1\varphi V=-\lambda \varphi V$ and taking $X=Y=X_{-\lambda}$ in \eqref{proof52} implies
$h_1\varphi V=\lambda \varphi V$. Since $V\neq 0$, this yields a contradiction.

Therefore, we can suppose from now on that \eqref{totumb} holds for $a\neq 0$. We deduce from equation \eqref{nablah2} that
\[
X(a) Y= \varphi \sigma(X,h_1 Y)+h_1 \varphi \sigma(X,Y)=g(X,h_1 Y)\varphi V+g(X,Y) h_1 \varphi V,
\]
for every $X,Y$ tangent vector fields.

Since $\dim N \geq 3$, we can take $Y$  linearly independent from $\varphi V$ and $h_1 \varphi V$. Then we deduce from the previous equation that $X(a)=0$, for every $X$, thus $a$ is a constant. Moreover, $g(X,h_1 Y)\varphi V+g(X,Y) h_1 \varphi V=0$, for every $X,Y$ tangent vector fields. Taking unit $X=Y$, we obtain that $h_1 \varphi V=-g(X,h_1X)\varphi V$, which is only possible if $h_1=0$ or $\varphi V=0$. If $h_1=0$, then substituting \eqref{totumb} in \eqref{nablah2} gives that $2ag(X,Y) \varphi V=0$, so again $\varphi V=0$.

In both cases, we have obtained that $\varphi V=0$, so $V$ is parallel to $\xi$  and it follows from $a=-\eta(V)$ that $V=-a \xi$ and $\sigma(X,Y)=-ag(X,Y)\xi$ holds, for every $X,Y$ tangent, where $a\neq 0$ is a constant.

Let us now recall Codazzi's equation \eqref{eqcodazzi}:
\[
(R(X,Y)Z)^\perp =(\nabla _{X}\sigma)(Y,Z)-(\nabla_{Y}\sigma)(X,Z).
\]
The first term is the normal component of $R(X,Y)Z$, so by equation \eqref{Rkappamu} and the fact that $h_2 X=h_1 X+a \varphi X$, we can write
\[
\begin{aligned}
(R(X,Y)Z)^\perp &= a (g(Y,Z) \varphi X -g(X,Z) \varphi Y)\\
&+a \frac{1-\frac\mu 2}{1-\kappa}  (g(h_1 Y,Z)  \varphi X-g(h_1 X,Z) \varphi Y)\\
&-a \frac{\kappa-\frac\mu 2}{1-\kappa}  (g(Y,Z)  \varphi h_1 X-g(X,Z) \varphi h_1 Y).
\end{aligned}
\]
On the other hand,
\begin{align*}
(\nabla _{X}\sigma)(Y,Z)
&=\nabla^\perp _{X}(\sigma(Y,Z))-\sigma(\overline\nabla _{X}Y,Z)-\sigma(Y,\overline\nabla _{X}Z)=\\
&=\nabla^\perp _{X}(-ag(Y,Z)\xi)+ag(\overline\nabla _{X}Y,Z)\xi+ag(\overline\nabla _{X}Z,X)\xi =\\
&=-ag(Y,Z)\overline\nabla _{X}^\perp \xi =ag(Y,Z) (\varphi X+\varphi h_1 X).
\end{align*}
Therefore, the second term of Codazzi's equation is
\begin{align*}
(\nabla _{X}\sigma)(Y,Z)&-(\nabla _{Y}\sigma)(X,Z)=ag(Y,Z) (\varphi X+\varphi h_1 X)-a g(X,Z) (\varphi Y+\varphi h_1 Y)\\
&=a (g(Y,Z) \varphi X -g(X,Z) \varphi Y)+a (g(Y,Z) \varphi h_1 X -g(X,Z) \varphi h_1 Y).
\end{align*}
Joining both terms, and bearing in mind that $a\neq0$, we obtain
\[
\begin{aligned}
& \frac{1-\frac\mu 2}{1-\kappa}  (g(h_1 Y,Z)  \varphi X-g(h_1 X,Z) \varphi Y)=\\
&= \frac{1-\frac\mu 2}{1-\kappa}  (g(Y,Z)  \varphi h_1 X-g(X,Z) \varphi h_1 Y).\\
\end{aligned}
\]
Since we are supposing that $I_M =\frac{1-\frac\mu 2}{\sqrt{1-\kappa}} \leq -1$, then $\frac{1-\frac\mu 2}{1-\kappa} \neq0$ and applying $\varphi$ to both terms of the previous equation gives us that
\begin{equation*}
g(h_1 Y,Z)   X-g(h_1 X,Z)  Y=g(Y,Z)   h_1 X-g(X,Z)  h_1 Y,\\
\end{equation*}
for every $X,Y,Z$ tangent vector fields.

Since $\dim (N) \geq 3$, we can choose $Y=Z$ unit and orthogonal to $X, h_1 X$, and we obtain that
\begin{equation}
h_1 X=g(h_1 Y,Y)X,\\
\end{equation}
and thus $h_1 X=bX$ for some function $b$.

From \eqref{propo11}, we have that $a^2+b^2=1-\kappa =\lambda^2 \neq0$, and in particular that $b$ must be constant. We can also write that $a=\lambda \cos(\theta)$ and $b=\lambda \sin(\theta)$ for some constant $\theta\in [-\pi,\pi]$. Since $a\neq0$, then $\theta \neq \pm \frac{\pi}{2}$.

By Gauss equation \eqref{eqgauss2} and the fact that $h_2 X=aX$, then
\begin{align*}
R(X,Y,Z,W)&=\overline R(X,Y,Z,W)-g(\sigma(X,W),\sigma(Y,Z))+g(\sigma(X,Z),\sigma(Y,W))=\\
&=\overline R(X,Y,Z,W)-a^2(g(X,W)g(Y,Z)+g(X,Z)g(Y,W)),
\end{align*}
for every $X,Y,Z,W$ tangent vector fields.

On the other hand, we know from equation \eqref{Rkappamu} and the fact that $hX =bX+a\varphi X$, that
\begin{align*}
R(X,Y,Z,W)= &\left(1-\frac\mu2 +2b+b^2 \frac{1-\frac\mu2}{1-\kappa}+a^2 \frac{\kappa-\frac\mu2}{1-\kappa} \right)\\
& \quad (g(X,W)g(Y,Z)-g(X,Z)g(Y,W)).
\end{align*}
Joining the last two equations, we obtain
\begin{align*}
\overline R(X,Y,Z,W)=&
\left(1-\frac\mu2 +2b+b^2 \frac{1-\frac\mu2}{1-\kappa}+a^2 \left(\frac{\kappa-\frac\mu2}{1-\kappa}+1 \right) \right) \\
&\quad (g(X,W)g(Y,Z)-g(X,Z)g(Y,W))\\
=&\left(1-\frac\mu2 +2b+(a^2+b^2) \frac{1-\frac\mu2}{1-\kappa}\right) \\
&\quad (g(X,W)g(Y,Z)-g(X,Z)g(Y,W))\\
=&2 (1 -\frac\mu2 +b ) (g(X,W)g(Y,Z)-g(X,Z)g(Y,W)).
\end{align*}
This means that  the submanifold is a space form with  constant curvature $2 (1 -\frac\mu2 +b )$. Moreover, since $I_M =\frac{1-\frac\mu 2}{\sqrt{1-\kappa}} \leq -1$ and $b=\lambda \sin (\theta) \neq \lambda$, then $1-\frac\mu 2+b < 1-\frac\mu 2+\lambda \leq0$ and the submanifold is a hyperbolic space $N=\mathbb{H}(2 (1 -\frac\mu2 +\lambda \sin(\theta) ))$.

Summing up, there exists $\theta \in [-\pi,\pi]$, $\theta \neq \pm \frac\pi2$, such that
\begin{equation}\label{dem2}
\begin{aligned}
N    &=\mathbb{H}(2 (1 -\frac\mu2 +\lambda \sin(\theta) )),\\
h_1 X&= \lambda \sin(\theta) X,\\
h_2 X&= \lambda \cos(\theta) X,\\
\sigma(X,Y)&=-\lambda \cos(\theta)g(X,Y) \xi.
\end{aligned}
\end{equation}

We have examples of submanifolds with these properties for every value of $\theta$. Indeed, Examples \ref{excd} with $c=\cos(\pi/4-\theta/2)$, $d=-\sin(\pi/4-\theta/2)$ satisfy
\[
\begin{aligned}
\sigma(cX_i+dY_i,cX_j+dY_j)
&=2\delta_{ij}cd\lambda\xi
=-2\delta_{ij} \sin\left(\frac\pi4-\frac\theta2 \right)
\cos \left(\frac\pi4-\frac\theta2 \right)\lambda \xi=\\
&=-\delta_{ij}\sin\left(\frac\pi2-\theta \right)\lambda \xi =-\delta_{ij}\lambda \cos(\theta)  \xi=\\
&=-\lambda \cos(\theta)g(cX_i+dY_i,cX_j+dY_j)  \xi,
\end{aligned}
\]
and the rest of conditions also hold.

Now, we will prove that any totally umbilical submanifold $N$ must be one of these, up to local isometries.
Let us denote by $F:N^n\to M^{2n+1}(\kappa,\mu)$ the immersion of $N$ into $M(\kappa,\mu)$. Since $\kappa<1$ and $I_M\leq -1$, we can suppose that, locally, $M(\kappa,\mu)$ is one of the Lie groups from Example \ref{books}. Thus, it is homogeneous and we can fix a point $p_0 \in N$ such that $F(p_0)=e$, where $e$ is the neutral element of the group.

We have that $N=\mathbb{H}(2 (1 -\frac\mu2 +\lambda \sin(\theta) ))$ and we can identify $N$ with its image as the (totally umbilical) integral submanifold through $e$ of the distribution spanned by $\{ cos \left(\frac{\pi}4-\frac{\theta}2 \right)  X_i(e)
-\sin \left(\frac{\pi}4-\frac{\theta}2 \right) Y_i(e)$, $i=1, \dots,n \}$. We denote by $G$ this immersion of $N$ and we take an orthonormal basis $\{e_1,\ldots,e_n\}$ at the point $p_0$ of $N$ such that
\[
dG(e_i) = cos \left(\frac{\pi}4-\frac{\theta}2 \right)  X_i(e)
-\sin \left(\frac{\pi}4-\frac{\theta}2 \right) Y_i(e), \; i=1, \dots,n.
\]

On the other hand, we have that
\begin{align}
h(dF(e_i))
 &=dF(\lambda \sin(\theta) e_i)+\varphi dF(\lambda \cos(\theta) e_i) \nonumber\\
 &= \lambda \sin(\theta) dF(e_i)+\lambda \cos(\theta) \varphi dF( e_i), \label{eq-h-dem}\\
h\varphi (d F(e_i))
  &=-\varphi h(dF(e_i))=\lambda \cos(\theta) dF(e_i)-\lambda \sin(\theta) \varphi dF( e_i). \label{eq-hphi-dem}
\end{align}
Therefore, using \eqref{eq-h-dem} and \eqref{eq-hphi-dem}, we can construct eigenvectors of $h$ associated with the eigenvalue $\lambda$ the following way:
\begin{align*}
h &\left(cos \left(\frac{\pi}4-\frac{\theta}2 \right)  dF(e_i)
+\sin \left(\frac{\pi}4-\frac{\theta}2 \right) \varphi ( dF(e_i)) \right)=\\
&=\lambda \left(
\left( cos \left(\frac{\pi}4-\frac{\theta}2 \right) \sin(\theta)
+\sin \left(\frac{\pi}4-\frac{\theta}2 \right) \cos(\theta) \right)
  dF(e_i) \right.\\
&
\hspace{0.6cm} \left.
+ \left(cos \left(\frac{\pi}4-\frac{\theta}2 \right) \cos(\theta)
-\sin \left(\frac{\pi}4-\frac{\theta}2 \right) \sin(\theta) \right)
\varphi (dF( e_i)) \right)=\\
&=\lambda \left(  \sin \left(\frac{\pi}4+\frac{\theta}2 \right)   dF(e_i)
+ \cos \left(\frac{\pi}4+\frac{\theta}2 \right)  \varphi ( dF( e_i) )\right)=\\
&=\lambda\left(cos \left(\frac{\pi}4-\frac{\theta}2 \right)  dF(e_i)
+\sin \left(\frac{\pi}4-\frac{\theta}2 \right) \varphi ( dF(e_i)) \right),
\end{align*}
for any $i=1,\ldots,n$.

Note that, by construction, both
\[
\cos \left(\frac{\pi}4-\frac{\theta}2 \right)  dF(e_i)
+\sin \left(\frac{\pi}4-\frac{\theta}2 \right) \varphi ( dF(e_i)), \; i=1,\dots,n
\]
and
\[
X_1(e),\dots,X_n(e)
\]
are basis of $E_e(\lambda)$. So, in view of Theorem 3 of \cite{boeckx2000}, there exists an isometry $H$ of $M^{2n+1}(\kappa,\mu)$ preserving the structure such that $H(e)=e$ and $H$ maps one basis of $E_e(\lambda)$ into the other one.
As a consequence, we have that $H \circ F(e) = G(e)$ and $d(H \circ F)(e_i) = dG(e_i)$.

We now take a geodesic $\gamma$ in $N$ through the point $p_0$. Since $N$ is totally umbilical with respect to both $H \circ F$ and $G$, then
$\gamma_1=H \circ F(\gamma)$ and $\gamma_2=G(\gamma)$ are curves in $M(\kappa,\mu)$ passing through $e$ that satisfy
$\nabla_{\gamma_1'} \gamma_1'=\nabla_{\gamma_2'} \gamma_2' =-\lambda \sin(\theta) \xi.$
Since $d(H \circ F)(e_i)=dG(e_i)$, they are also determined by the same initial conditions. Therefore, both curves need to coincide, so  $H \circ F(\gamma(s))= G(\gamma(s))$ for all $s$ and thus $F$ and $G$ are congruent.
\end{proof}


\end{document}